\documentclass[11pt]{article}
\usepackage{cite}
\usepackage{mathrsfs}
\usepackage{amsmath, amssymb, amsthm, latexsym, amsfonts, epsfig, color,graphicx,authblk}
\allowdisplaybreaks[4]
\numberwithin{equation}{section}
\begin{document}

\newcommand{\s}{\sigma}
\renewcommand{\k}{\kappa}
\newcommand{\p}{\partial}
\newcommand{\D}{\Delta}
\newcommand{\om}{\omega}
\newcommand{\Om}{\Omega}
\renewcommand{\phi}{\varphi}
\newcommand{\e}{\epsilon}
\renewcommand{\a}{\alpha}
\renewcommand{\b}{\beta}
\newcommand{\N}{{\mathbb N}}
\newcommand{\R}{{\mathbb R}}
   \newcommand{\eps}{\varepsilon}
   \newcommand{\EX}{{\Bbb{E}}}
   \newcommand{\PX}{{\Bbb{P}}}

\newcommand{\cF}{{\cal F}}
\newcommand{\cG}{{\cal G}}
\newcommand{\cD}{{\cal D}}
\newcommand{\cO}{{\cal O}}

\newcommand{\de}{\delta}

\newcommand{\grad}{\nabla}
\newcommand{\n}{\nabla}
\newcommand{\curl}{\nabla \times}
\newcommand{\dive}{\nabla \cdot}

\newcommand{\ddt}{\frac{d}{dt}}
\newcommand{\la}{{\lambda}}

\newtheorem{theorem}{Theorem}[section]
\newtheorem{lemma}{Lemma}[section]
\newtheorem{remark}{Remark}[section]
\newtheorem{example}{Example}[section]
\newtheorem{definition}{Definition}[section]
\newtheorem{corollary}{Corollary}[section]
\newtheorem{assumption}{Assumption}[section]
\newtheorem{prop}{Proposition}[section]
\newtheorem{notation}{Notation}[section]
\def\proof{\mbox {\it Proof.~}}
\makeatletter
\@addtoreset{equation}{section}
\makeatother
\renewcommand{\theequation}{\arabic{section}.\arabic{equation}}

 \makeatletter\def\theequation{\arabic{section}.\arabic{equation}}\makeatother

\title{ On the limit distribution for  stochastic differential equations driven by cylindrical non-symmetric $\alpha$-stable L\'{e}vy processes}
\author[a]{\textbf{Ting Li}}
\author[b]{\textbf{Hongbo Fu} \thanks {Corresponding author: hbfu@wtu.edu.cn}}
\author[a]{\textbf{Xianming Liu}}
\affil[a]{School of Mathematics and Statistics,
\authorcr \textit{Huazhong University of Science and Technology, Wuhan 430074, P. R. China.}}
\affil[b]{Research Center for Applied Mathematics and Interdisciplinary Sciences,
	\authorcr \textit{School of Mathematical and Physical Sciences,}
	\authorcr \textit{Wuhan Textile University, Wuhan 430073, P. R. China.}}
\renewcommand{\Affilfont}{\small\it}
\renewcommand\Authands{ \textbf{and} }
 \date{\today }
\maketitle

\begin{abstract}
This article deals with the limit distribution for  a stochastic differential equation
driven by a non-symmetric cylindrical $\alpha$-stable process. Under suitable conditions, it is proved that
the solution of this equation converges weakly to that of a  stochastic differential equation driven by a  Brownian motion
in the Skorohod space as $\alpha\rightarrow2$.  Also, the rate of  weak convergence, which depends on $ 2-\alpha$, for the
solution  towards the solution of the limit equation is obtained. For illustration, the results are applied to a  simple one-dimensional stochastic
differential equation, which implies the rate of  weak convergence is optimal.

\end{abstract}

\noindent {\it \footnotesize Key words}. {\scriptsize
$\alpha$-stable L\'{e}vy process;  Tightness;
Rate of weak convergence; Backforward Kolmogorov equation}

%



\setcounter{secnumdepth}{5} \setcounter{tocdepth}{5}

\makeatletter
    \newcommand\figcaption{\def\@captype{figure}\caption}
    \newcommand\tabcaption{\def\@captype{table}\caption}
\makeatother


%


\section{\bf Introduction} \label{pre}

Stochastic differential equations (SDEs) driven by Wiener noise and L\'{e}vy noise have important applications to finance and physics. It is well known that Brownian motion is an $\alpha$-stable L\'{e}vy process with $\alpha=2$.  For one-dimensional stochastic process, we know that an $\alpha$-stable process converges in distribution to a Brownian motion as $\alpha\rightarrow2$ (see Sato \cite{Sato}).  Then a natural question arises: for the  SDE driven by an $\alpha$-stable process, dose its solution  converge in distribution to that of an SDE driven by a Brownian motion?

Liu \cite{Liu} studied weak convergence for the solution of an SDE driven by the symmetric stable process, and the $\alpha$-continuity of the solution to an SDE driven by an $\alpha$-stable process as $\alpha$ tends to $\alpha^* \in (1,2]$, but the case of non-symmetric $\alpha$-stable process  is not  well studied. In this paper, we study the weak convergence behavior as $\alpha \rightarrow 2$  for the solutions of SDEs driven by non-symmetic $\alpha$-stable processes in the Skorokhod space $D_{J_1}([0,T], \mathbb{R}^d$) (the space $D([0,T], \mathbb{R}^d)$ equipped with the $J_1$-metric, also see \cite{Bill}).

Let  $(\Omega,\mathcal{F},\{\mathcal{F}_t\}_{t\geq 0},\mathbb{P})$  be
a filtered probability space, which satisfies the usual conditions.  Consider the following equations,
\begin{align}
dX^{\alpha,\beta} _t&=b(X^{\alpha,\beta} _t)dt+ \sigma (X^{\alpha,\beta} _{t-}) dL^{\alpha,\beta}_t,
\quad X^{\alpha,\beta}_0=x\in \mathbb{R}^d,   \label{e:0} \\
dX _t&=b(X_t)dt+ \sqrt{2}\sigma (X_{t}) dB_t,\quad  X_0=x, \label{e2}
\end{align}
where $b(x)=(b_1(x),\cdot\cdot\cdot,b_d(x))$  is a vector function, $\sigma(x)=(\sigma_{ij}(x))_{d\times d}$ is a $d\times d$ matrix-valued function. For more assumptions on the drift  and diffusion coefficients,  see Assumption \ref{asp1} and Assumption\ref{asp2}. The noise $B_t=(B_t^1, \cdots, B_t^d)$ is a $d$-dimensional Brownian motion, and $L^{\alpha,\beta}_t=(L^{\alpha,\beta_1}_t,\cdots,L^{\alpha,\beta_d}_t)$ is a $d$-dimensional cylindrical  non-symmetric $\alpha$-stable process. For each $i\in\{1,2,...,d\}$,  $L^{\alpha,\beta_i}_t$ is an one-dimensional non-symmetric $\alpha$-stable L\'{e}vy process and $N^{\alpha,\beta_i}$ is an $\mathcal{F}_t$-Poisson random measure with intensity measure $\nu^{\alpha,\beta_i}(dz)dt$. The  L\'{e}vy measure $\nu^{\alpha,\beta_i}$ on $\mathbb{R}$  is defined as
\begin{equation}\label{e1}
\nu^{\alpha,\beta_i}(dz)=\left\{
\begin{array}{rcl}
\frac{C_1^idz}{|z|^{1+\alpha}},&&{z > 0},\\
\ \\
\frac{C_2^i dz}{|z|^{1+\alpha}},&&{z < 0},
\end{array}\right.
\end{equation}
with $\beta_i\in [-1,1]$, \  $C_1^i=K_\alpha^i\frac{1+\beta_i}{2}$ and $C_2^i=K_\alpha^i\frac{1-\beta_i}{2}$, where
\begin{equation}\nonumber
K_\alpha^i=\left\{
\begin{array}{rcl}
\frac{\alpha(1-\alpha)}{\Gamma(2-\alpha)cos(\frac{\pi\alpha}{2})},&&{\alpha \neq 1},\\
\ \\
\frac{2}{\pi}\             ,&&{\alpha=1}.
\end{array}\right.
\end{equation}
In particular, $\nu^{\alpha,\beta_i}$ is said to be symmetric if $\beta_i = 0$.
Furthermore, the compensated Poisson random measure $\tilde{N}^{\alpha,\beta_i}$ is defined
by
\[
\tilde{N}^{\alpha,\beta_i}(dt,dz):=N^{\alpha,\beta_i}(dt,dz)- \nu^{\alpha,\beta_i}(dz)dt.
\]

 According to
L\'{e}vy-It\^{o}'s decomposition,  we have
\begin{equation*}\label{levy-ito-1}
L^{\alpha,\beta_i}_t =  \int_0^t \int_{|z|\leq1} z
\tilde{N}^{\alpha,\beta_i}(ds,dz) + \int_0^t \int_{|z|>1} z
N^{\alpha,\beta_i}(ds,dz),\quad \alpha \in (1,2).
\end{equation*}
For every $i\in\{1,2,...,d\}$,
\begin{equation*}\label{e1-eqv}
dX^{\alpha,\beta_i} _t=b_i(X^{\alpha,\beta}_t)dt+\sum_{j=1}^d\int_{|z_j|\leq 1}\sigma_{ij}
(X^{\alpha,\beta}_{t-}) z_j \tilde{N}^{\alpha,\beta_j}(dt,dz_j) +\sum_{j=1}^d \int_{|z_j|>1}\sigma_{ij}(X^{\alpha,\beta}_{t-}) z_j
N^{\alpha,\beta_j}(dt,dz_j),
\end{equation*}
is the $i$th component  of SDE \eqref{e:0}.

We will prove the solution of  Equation \eqref{e:0}, whose noise is non-symmetric $\alpha$-stable, converges weakly to the solution of Equation \eqref{e2} in Skorohod space (see Section 3), which we prove by using the  Aldous' criterion.  To this purpose
we need to establish the tightness of the approximating solution  and this can be obtained by a priori estimate
 $ \sup\limits_{\alpha\geq\alpha_0} \EX [\sup\limits_{0\leq t\leq T} |X^{
    \alpha,\beta}_t|^\theta] <\infty. $
Secondly, by martingale characterization, we show that the limit (in the sense of distributions) of the solution of the approximating equation  is the solution of an SDE driven by a Brownian motion.

While the weak convergence properties are established for SDEs driven by non-symmetric $\alpha$-stable processes in Skorohod space, the rate of convergence is still not easily derived. To our knowledge, many works are devoted to study the convergence rate of the solution  for an  SDE, where
both the approximation equation and the limit equation taking the same type of noise term (see, for instance, \cite{rate} and \cite{rate1}).  Here we will derive
 the explicit error bounds on the difference between the solution of  \eqref{e:0}  and the solution of   limit equation \eqref{e2}, where the two noises appearing in above two equations are
 totally different.  To this
purpose, we introduce the backward Kolmogorov equation associated with the solution of SDE \eqref{e2} to construct
an evolutionary equation that describes the  weak difference. Furthermore, we need to ensure that the solution of SDE \eqref{e:0} converges in weak sense to the solution of SDE \eqref{e2} under Assumption \ref{asp2}. For simplicity, we only establish the uniform estimates on solution  of approximating equation  (see Lemma \ref{lem4.2}).  And the corresponding proof of tightness and martingale characterization is omitted. Finally,  to illustrate the estimate \eqref{wr} is the optimal, we present a special case in Example \ref{example}.

The organization of this paper is given as follows. We introduce some notations,  basic settings and state the main results in Section 2.  Section 3 contains the proof of the main results. Section 4 is devoted to the rate of weak convergence and an
illustrative example.

\section{\bf Notation and main results}

Denote $H=[0,T]\times \mathbb{R}^d,\  \mathbb{N}=\{0, 1, 2, ...\}$.  For $(t,x)$ in $H$, $\eta=(\eta_1, \cdots, \eta_d)$ in  $(\mathbb{N}^+)^d$ is a multi-index, $|\eta|=\eta_1 +\cdots+ \eta_d$, and $D^{\eta}
=\frac{\partial^{|\eta|}}{\partial x_1^{\eta_1}\cdots\partial x_d^{\eta_d}}$. 	We denote
\begin{align}\nonumber
&\partial_t u(t,x)=\frac{\partial}{\partial t}u(t,x), \ \ \ \ \  D^ku(t,x)=(D^{\eta}u(t,x))_{|\eta|=k}, k\in\mathbb{N},\\ \nonumber
&\partial_i u(t,x)=u_{x_i}(t,x), \ \ \ \ \  \partial^2_{ij} u(t,x)=u_{x_i x_j}(t,x), \ \  i, j \in \{1,\cdots, d\} \\ \nonumber
&\partial_xu(t,x)=\nabla u(t,x)=\nabla_x u(t,x)=(\partial_1 u(t,x),\cdots, \partial_du(t,x)). \nonumber
\end{align}
Let $\kappa \in  \mathbb{N}^+$ and $0< \gamma \le 1$. And let $ C^{\kappa, \gamma}(H)$ denote the space of measurable functions in $H$ such that  the derivatives $D_x^\eta u(t,x)$ are continuous with respect to $(t,x)$ for all $1\le|\eta|\le\kappa$ and  the derivatives are locally $\gamma-$H${\rm \ddot{o}}$lder continuous with respect to $x$ for $|\eta|=\kappa$. We denote its norm by
\begin{align} \nonumber
||u||_{\kappa,\delta}:=\sum_{ |\eta|\le \kappa } |D_x^\eta u(t,x)|_0+\sup_{|\eta|=\kappa  ,
\atop t,\ x \neq x^1 }   \frac{|D_x^\eta u(t,x)- D_x^\eta u(t,x)|}{|x-x^1|^{\gamma}}, \nonumber
\end{align}
where $|u|_0 = \sup_{(t,x)\in H}|u(t,x)|$. We denote by $C^{\kappa, \gamma}(\mathbb{R}^d)$ the corresponding function sapce on $\mathbb{R}^d$.

The letter $C=C(\cdot, ... , \cdot)$  is a constant, which only depends on quantities appearing in parentheses. In a given context, the same letter is (generally) used to denote different constants depending on the same set of arguments. Moreover, all of the constant $C$ in this study are independent of $\alpha$, but it possibly depends on $\beta$.

We denote by $H(f)$  the Hessian matrix of $f$ and let $H_x(u(t,x))$ be the Hessian matrix of $u$ with respect to $x$. Let $\sigma^*$  be the transpose of $\sigma$,  and set
\begin{equation}\nonumber
\sigma=(\sigma_1,\cdots, \sigma_d),\ \  \sigma_i=(\sigma_{1i},\cdots,\sigma_{di})^* .
\end{equation}
We use the Einstein summation convention throughout this paper, for example if $x,y$ in $\mathbb{R}^d$, \ then
\begin{equation}\nonumber
x y =\sum_{i=1}^d x_iy_i.
\end{equation}

 Now, we make some assumptions on the drift and diffusion coefficients.

 \begin{assumption} \label{asp1}
(a) There is a constant $K_{\sigma}$ such
that for every $x,y \in \mathbb{R}^d$,
\begin{align*}
\|\sigma(x)-\sigma(y)\|_{HS}&\leq K_\sigma|x-y|,
\end{align*}
where $||\cdot||_{\text{HS}}$ denotes the Hilbert-Schmidt norm of a
matrix.

(b) There exists a positive number $\delta_0$, such that for all $R>0$, and $x,y\in\mathbb{R}^d$, with $\max{\{|y|,|x|\}}\le R$
and $|y-x|\le\delta_0$, satisfying
\begin{equation}\nonumber
\langle y-x , b(y)-b(x) \rangle \le K_{b} |y-x|^2, \nonumber
\end{equation}
where $K_b$ is a positive constant. Also,  there exists a constant $C>0$, such that for all $x\in\mathbb{R}^d$, satisfying
\begin{equation}\nonumber
|b(x)|\le C(1+|x|^r), \nonumber
\end{equation}
where $r\in(0,+\infty)$.
\end{assumption}

Observe that in view of the above assumptions, drift $b$ is dissipative and locally bounded. The second item ensures the tightness of the approximating solution (see Lemma \ref{lem2.4}), which is   similar to that  in \cite{SGI}.  With our assumptions we have the following main result which is proved later in section 3.
\begin{theorem}\label{thm1}
If  Assumption \ref{asp1} is satisfied. Let $X_t^{\alpha,\beta}$ be the solution of SDE \eqref{e:0}, then there
exists a measure $\mu$ defined on $D([0,T],\mathbb{R}^d)$ so that
\begin{equation}\label{main1} \nonumber
\mathcal{L} (X^{\alpha,\beta}) \rightarrow  \mu \quad \text{as} \quad
\alpha\rightarrow 2,
\end{equation}
in $D_{J_1}([0,T],\mathbb{R}^d)$. Moreover, $\mu$ is the law of
solution to SDE \eqref{e2}.

\end{theorem}

Note that this theorem is different from \cite[Theorem 1.1]{Liu},  where the noise is symmetric and rotationally invariant $\alpha$-stable L\'{e}vy process and the drift coefficient satisfies global Lipschitz condition, while in our setting the noise is cylindrical non-symmetric $\alpha$-stable L\'{e}vy process and the drift is non-Lipschitz.  In order to investigate the convergence rate, we need to impose the following assumptions on the coefficients $b$ and $\sigma$.

\begin{assumption}\label{asp2}
Let $0< \gamma \le 1$, $b_i(x) \in {C}^{2, \gamma}(\mathbb{R}^d)$, $\sigma_{ij}(x) \in  {C}^{2, \gamma}(\mathbb{R}^d)$, for all $1\le i,j \le d.$
\end{assumption}

The above assumptions are similar to the conditions imposed on the drift and jump coefficients in \cite{rate}. And it ensures the existence and uniqueness for the backward Kolmogorov equation which is used  to derive the rate of weak convergence.  We  state the following theorem.

%

\begin{theorem} \label{thm2}
If  Assumption \ref{asp2} is satisfied. Let $X_t$ be the solution of SDE \eqref{e2}, then for each $f\in {C}^{2,\gamma}(\mathbb{R}^d)$, we have
\begin{equation} \label{wr}
\EX[f(X_t^{\alpha,\beta})-f(X_t)]\le C(2-\alpha).
\end{equation}
\end{theorem}

This is our second main result and whose proof will be carried out in Section \ref{Weak Convergence}.

\section{\bf Proof of Theorem 2.1 } \label{proof1}

This section is devoted to the proof of Theorem \ref{thm1}. First, we state an existence and uniqueness result for SDE \eqref{e:0}. Then, by using uniform estimates of the approximating solution  (Lemma \ref{lem2.3}) we prove the tightness of the approximating equation (Lemma \ref{lem2.4}). Finally, by martingale characterization, we show that the limit of the solution  of the approximating equation  is the solution of an SDE driven by a Brownian motion. Now we are ready to state that SDE \eqref{e:0} has a unique strong solution.

\begin{lemma}\label{lem2.1}
If  Assumption \ref{asp1} is satisfied, then SDE \eqref{e:0}
admits a unique strong solution.
\end{lemma}

The proof of this result is similar to \cite[Corollary 2.9]{zhu}, where a  standard truncated method is  employed to  show that the stochastic differential equation has a unique strong solution, and hence is omitted here. In this article, the following lemma plays a crucial role.

\begin{lemma}\label{lem2.2}
For every  $0<\delta\le 1$, $l \in (1,2)$,
 $\vartheta \in [1,l)$, and $\nu^{\alpha,\beta}$ define as \eqref{e1}, we have
\item{\begin{equation}  \label{eq3.1}
 \sup_{\alpha\geq l} \int_{ |z|\leq \delta} |z|^{2}
  \nu^{\alpha,\beta}(dz)+\sup_{\alpha\geq l}
  \int_{ |z|> \delta} |z|^\vartheta \nu^{\alpha,\beta}(dz) = C<\infty
\end{equation}}
and
\item{\begin{equation} \label{eq3.2}
\lim_{\alpha \rightarrow 2}\int_{ |z|\leq \delta}|z|^2
\nu^{\alpha,\beta}(dz)=\lim_{\alpha \rightarrow 2}\frac{
K_\alpha}{2-\alpha}= 2.
\end{equation}}
For every $\delta>0$ and $ \vartheta <2$, we have
\item{\begin{equation} \label{eq3.3}
\lim_{\alpha \rightarrow 2}\int_{ |z|>\delta}|z|^\vartheta
\nu^{\alpha,\beta}(dz) =0.
\end{equation}}
For every $\alpha$ in $(0,2)$, we have
\item{\begin{equation} \label{eq3.4}
 \lim_{\delta\rightarrow 0+}\delta^{\alpha-2}\int_{ |z|\leq \delta} |z|^{2}
  \nu^{\alpha,\beta}(dz)=c <\infty, \quad \lim_{\delta\rightarrow 0+}\int_{ |z|\leq \delta} |z|^{2}
  \nu^{\alpha,\beta}(dz)=0.
\end{equation}}
Moreover, let $\alpha\in(\frac{3}{2}, 2)$, we have
\item{\begin{equation} \label{eq3.5}
|\int_{|z|\le\delta} |z|^2 \nu^{\alpha,\beta}(dz) - 2| + \int_{|z|>\delta}  |z| \nu^{\alpha,\beta}(dz) \le C(2-\alpha).
\end{equation}}
\end{lemma}

We give a proof for this result in Section \ref{Appendix}. Next, we derive a uniform estimate in the following lemma.

\begin{lemma}\label{lem2.3}
If  Assumption \ref{asp1} is satisfied, then
\begin{equation}
    \sup_{\alpha\geq\alpha_0} \EX [\sup_{0\leq t\leq T} |X^{
    \alpha,\beta}_t|^\theta] <\infty. \label{solution0}
\end{equation}
 where  $ \theta \in (0,1)$ and $\alpha_0 \in (\frac{3}{2} , 2)$.
\end{lemma}
\begin{proof}  Let $f(x)=(1+|x|^2)^\frac{1}{2}$ and $\delta>0$,
by using It\^{o} formula,  we have
\begin{align}\label{ito formula}
df(X^{\alpha,\beta}_t)&=\quad b(X^{\alpha,\beta}_t)\cdot \nabla f(X^{\alpha,\beta}_t)dt  \nonumber \\
&\quad+\sum_{i=1}^{d}\int_{ |z_i|>\delta}[f(X^{\alpha,\beta}_{t-}+
\sigma_i(X^{\alpha,\beta}_{t-})z_i)-f(X^{\alpha,\beta}_{t-})]
\nu^{\alpha,\beta_i}(dz_i)dt\nonumber \\
&\quad +\sum_{i=1}^d\int_{|z_i|\leq \delta}[f(X^{\alpha,\beta}_{t-}+
\sigma_i(X^{\alpha,\beta}_{t-})z_i)-f(X^{\alpha,\beta}_{t-})\nonumber \\
&\quad\quad\quad\quad\quad\quad\quad- \sigma_i(X^{\alpha,\beta}_{t-})z_i\cdot
\nabla f(X^{\alpha,\beta}_{t-})] \nu^{\alpha,\beta_i}(dz_i)dt\nonumber \\
&\quad-\sum_{i=1}^d\int_{\delta<|z_i|\leq 1}\sigma_i(X^{\alpha,\beta}_{t-})z_i\cdot
\nabla f(X^{\alpha,\beta}_{t-})] \nu^{\alpha,\beta_i}(dz_i)dt\nonumber \\
&\quad \  +\sum_{i=1}^d\int_{ \mathbb{R}-\{0\}}[f(X^{\alpha,\beta}_{t-}+
\sigma_i(X^{\alpha,\beta}_{t-})z_i)-f(X^{\alpha,\beta}_{t-})]
\tilde{N}^{\alpha,\beta_i}(dt,dz_i)\nonumber \\
&:=  I_1(X_t^{\alpha,\beta})dt+I_2(X_t^{\alpha,\beta})dt+I_3(X_t^{\alpha,\beta})dt+I_4(X_t^{\alpha,\beta})+dM_t.
\end{align}
For $I_1(X_t^{\alpha,\beta})$, we have
\begin{equation} \label{I1}
|I_1(X_t^{\alpha,\beta})|\le |b_i(X_t^{\alpha,\beta})\partial^i f(X_t^{\alpha,\beta})|=\frac{|\langle X_t^{\alpha,\beta}, b(X_t^{\alpha,\beta})\rangle|}{(1+|X_t^{\alpha,\beta}|^2)^{\frac{1}{2}}}\le C f(X_t^{\alpha,\beta}).
\end{equation}
We observe that
\begin{align*}
|f(x+y)-f(x)|\le|y|\int_0^1|\nabla f(x+sy)|ds\le \frac{|y|}{2}.
\end{align*}
Using Lemma \ref{lem2.2} and the above inequality, we deduce that
\begin{align}  \label{I2}
|I_2(X_t^{\alpha,\beta})|&\le\sum_{i=1}^{d}\int_{|z_i|>\delta}|f(X_{t-}^{\alpha,\beta} +\sigma_i(X_{t-}^{\alpha,\beta})z_i)-f(X_{t-}^{\alpha,\beta})|\nu^{\alpha,\beta_i}(dz_i) \nonumber \\
&\le C\sum_{i=1}^{d}\int_{|z_i|>\delta}|\sigma_i(X_{t-}^{\alpha,\beta})z_i|\nu^{\alpha,\beta_i}(dz_i)\nonumber \\
&\le C  \sum_{i=1}^{d}\int_{|z_i|>\delta}|z_i|\nu^{\alpha,\beta_i}(dz_i)(|X_t^{\alpha,\beta}|+1)\nonumber \\
&\le C (f(X_t^{\alpha,\beta})+1) .
\end{align}
For $I_3(X_{t}^{\alpha,\beta})$, we can choose a sufficiently small $\delta>0$ such that for all  $|z|\le\delta$,
\begin{align*}
&\quad|f(X^{\alpha,\beta}_{t-}+
\sigma_i(X^{\alpha,\beta}_{t-})z_i)-f(X^{\alpha,\beta}_{t-})
- \sigma_i(X^{\alpha,\beta}_{t-})z_i\cdot
\nabla f(X^{\alpha,\beta}_{t-})|\\
&=|\int_0^1dr \int_0^r  (\sigma_i(X_{t-}^{\alpha,\beta})z_i)H(f)(X^{\alpha,\beta}_{t-}+s\sigma_i(X_{t-}^{\alpha,\beta})z_i) \cdot(\sigma_i(X_{t-}^{\alpha,\beta})z_i)ds|\\
&\le \sup_{s\le1}\frac{C|\sigma_i(X_{t-}^{\alpha,\beta})z_i|^2}{(1+|X_{t-}^{\alpha,\beta} +s\sigma_i(X_{t-}^{\alpha,\beta})z_i|^2)^{\frac{1}{2}}}\\ \nonumber
&\le \frac{C|\sigma_i(X_{t-}^{\alpha,\beta})z_i|^2}{(1+|X_{t-}^{\alpha,\beta}|^2)^{\frac{1}{2}}}.
\end{align*}
In the last step, since
\begin{align*}
&\quad\frac{1}{1+|X_{t-}^{\alpha,\beta} +s\sigma_i(X_{t-}^{\alpha,\beta})z_i|^2}\\
&\le \frac{1}{1+|\frac{X_{t-}^{\alpha,\beta}}{2}+2s\sigma_i(X_{t-}^{\alpha,\beta})z_i|^2+\frac{3|X_{t-}^{\alpha,\beta}|^2}{4}-3|s\sigma_i(X_{t-}^{\alpha,\beta})z_i|^2}\\
&\le \frac{C}{1+|X_{t-}^{\alpha,\beta}|^2}.
\end{align*}
Thus, we have
\begin{align} \label{I3}
&|I_3(X_{t}^{\alpha,\beta})|\le C\sum_{i=1}^d\int_{|z_i|\le \delta} \frac{|\sigma_i(X_{t-}^{\alpha,\beta})z_i|^2}{(1+|X_{t-}^{\alpha,\beta}|^2)^{\frac{1}{2}}}\nu^{\alpha,\beta_i}(dz_i) \nonumber \\
&\le C \sum_{i=1}^d\int_{|z_i|\le \delta} |z_i|^2\nu^{\alpha,\beta_i}(dz_i)(1+|X_{t-}^{\alpha,\beta}|^2)^{\frac{1}{2}}\nonumber \\
&\le C f(X_{t}^{\alpha,\beta}).
\end{align}
For $I_4(X_{t}^{\alpha,\beta})$, we have
\begin{align}\label{I4}
|I_4(X_{t}^{\alpha,\beta})|&\le \sum_{i=1}^d \int_{\delta \le |z_i| < 1} |\sigma_i(X^{\alpha,\beta}_{t-})z_i\cdot
\nabla f(X^{\alpha,\beta}_{t-})| \nu^{\alpha,\beta_i}(dz_i)\nonumber \\
&\le C \sum_{i=1}^d \int_{|z|>\delta} |z_i|f(X^{\alpha,\beta}_{t-}) \nu^{\alpha,\beta_i}(dz_i) \le Cf(X^{\alpha,\beta}_{t-}).
\end{align}
From   \eqref{ito formula}-\eqref{I4}, we get
\begin{equation}
f(X_t^{\alpha,\beta})-f(x_0)\le C\int_0^t f(X_s^{\alpha,\beta})ds+C+M_t. \nonumber
\end{equation}
By using a stochastic Gronwall's inequality in Xie and Zhang \cite[lemma 3.8]{SGI}, for any\ $1<q<\infty$, we obtain
\begin{align*}
\EX \sup_{0\le t \le T}|X_t^{\alpha,\beta}|^\frac{1}{q}
&\le \EX \sup_{0\le t \le T}{f(X_t^{\alpha,\beta})}^\frac{1}{q}\\
&\le C(q,T)<\infty.
\end{align*}
Since the inequality always holds for all $q\in(1,\infty)$, then we have
\begin{equation}\label{solution0} \nonumber
    \sup_{\alpha\geq\alpha_0} \EX [\sup_{0\leq t\leq T} |X^{
    \alpha,\beta}_t|^\theta] <\infty, \nonumber
\end{equation}
where $\theta\in(0,1)$.\qed\end{proof}\bigskip

The result of uniform estimate can be used below to establish the tightness of the approximating solution.

\begin{lemma}\label{lem2.4}
If  Assumption \ref{asp1} is satisfied. The family
$\{X^{\alpha,\beta}\}_{\alpha \geq \alpha_0} $ is tight in the space
$D_{J_1}([0,T],\mathbb{R}^d)$.
\end{lemma}
\begin{proof}
According to Aldous' criterion \cite[Thereom 1.1]{Aldous}, we only  prove the following two items:

\((1)\)  For every $\varepsilon >0$, there exists $L_\varepsilon$ such that
\begin{equation} \label{Aldous1}
    \sup_{\alpha \geq \alpha_0}\PX(\sup_{0\leq t\leq T}|X^{\alpha,\beta}_t| >L_\varepsilon) \leq
\varepsilon.
\end{equation}

\((2)\) For every stopping time $\tau^\alpha \in [0,T]$, and every $K>0$, we
have
\begin{equation}\label{Aldous2}
    \lim_{\varepsilon\rightarrow 0}
\sup_{\alpha \geq
\alpha_0}\PX(|X^{\alpha,\beta}_{T\wedge(\tau^\alpha+\varepsilon)}-X^{\alpha,\beta}_{\tau^\alpha}|
>K) =0.
\end{equation}

By using Chebyshev's inequality, we have
\begin{align}\label{Che}
\sup_{\alpha \geq \alpha_0}\PX(\sup_{0\leq t\leq T}|X^{\alpha,\beta}_t| >L) \leq\frac{1}{L^\theta} \sup_{\alpha>\alpha_0} \EX(\sup_{0\le t\e T}|X_t^{\alpha}|^{\theta})\le\frac{C}{L^\theta},
\end{align}
where $\theta \in (0, 1)$, and $\alpha_0 \in (\frac{3}{2}, 2)$. This  immediately implies the estimate \eqref{Aldous1}.

For notational convenience, we set
$\tau^\alpha+\varepsilon:=T\wedge(\tau^\alpha+\varepsilon)$ and
obtain that
\begin{align*}
 \quad\sup_{\alpha \geq \alpha_0}\PX(|X^{\alpha,\beta}_{\tau^\alpha+\varepsilon}-X^{\alpha,\beta}_{\tau^\alpha}|^2
>K)& \leq  \sup_{\alpha \geq
\alpha_0}\PX(|\int_{\tau^\alpha}^{\tau^\alpha+\varepsilon}
b(X^{\alpha,\beta}_s)ds|>\frac{K}{3})\\
&\quad+ \sup_{\alpha \geq
\alpha_0}\PX(|\int_{\tau^\alpha}^{\tau^\alpha+\varepsilon}
\int_{ |z|\leq 1} z\sigma(X^{\alpha,\beta}_{s-}) \tilde{N}^{\alpha,\beta}(ds,dz)|>\frac{K}{3})\\
&\quad+ \sup_{\alpha \geq
\alpha_0}\PX(|\int_{\tau^\alpha}^{\tau^\alpha+\varepsilon}
 \int_{ |z|>1} z\sigma(X^{\alpha,\beta}_{s-}) N^{\alpha,\beta}(ds,dz)|>\frac{K}{3})\\
 &:= K_1+K_2+K_3.
\end{align*}
Let $\ p \in(0,1)$, for $K_1$, by using inequality $(a+b)^p \le (a^p + b^p)$ and Lemma \ref{lem2.3} we have
\begin{align*}
   K_1
&\leq \sup_{\alpha \geq
\alpha_0}\PX(|\int_{\tau^\alpha}^{\tau^\alpha+\varepsilon}
C(1+|X^{\alpha,\beta}_s|^r)ds|>\frac{K}{3}) \\
&\leq \sup_{\alpha \geq
\alpha_0}\frac{C}{K^{\frac{1}{r+1}}}\EX|\int_{\tau^\alpha}^{\tau^\alpha+\varepsilon}
(1+|X^{\alpha,\beta}_s|^r)ds|^{\frac{1}{r+1}} \\
&\leq\frac{C\varepsilon^{\frac{1}{r+1}}}{K^{\frac{1}{r +1}}} \sup_{\alpha \geq
\alpha_0} \EX (1+ \sup_{0<s \leq T}|X^{\alpha,\beta}_s|^{\frac{r}{r +1}})\rightarrow 0,
\quad \text{as} \quad \varepsilon \rightarrow 0.
\end{align*}
By Burkholder's
inequality, we obtain that
\begin{align*}
   K_2
&=\sup_{\alpha \geq
\alpha_0}\PX(|\int_{\tau^\alpha}^{\tau^\alpha+\varepsilon}
 \int_{ |z|\leq 1} z\sigma( X^{\alpha,\beta}_{s-}) \tilde{N}^{\alpha,\beta}(ds,dz)|>\frac{K}{3})\\
&\leq \sup_{\alpha \geq
\alpha_0}\PX(|\int_{\tau^\alpha}^{\tau^\alpha+\varepsilon}
 \int_{ |z|\leq 1} z\sigma( X^{\alpha,\beta}_{s-}) \tilde{N}^{\alpha,\beta}(ds,dz)|>\frac{K}{3}; \sup_{0\leq t\leq T}|X^{\alpha,\beta}_t|\leq M)\\
 &\quad+  \PX(\sup_{0\leq t\leq T}|X^{\alpha,\beta}_t|\geq M)\\
&\leq \frac{3}{K} \sup_{\alpha \geq
\alpha_0}\EX|(\int_{\tau^\alpha}^{\tau^\alpha+\varepsilon} \int_{
|z|\leq
   1} z\sigma(X^{\alpha,\beta}_{s-}) \tilde{N}^{\alpha,\beta}(ds,dz); \sup_{0\leq t\leq T}|X^{\alpha,\beta}_t|\leq
   M)| \\
  &\quad +  \PX(\sup_{0\leq t\leq T}|X^{\alpha,\beta}_t|\geq M)\\
   &\leq \frac{C}{K} \sup_{\alpha \geq \alpha_0}[\EX
  \int_{\tau^\alpha}^{\tau^\alpha+\varepsilon} \int_{ |z|\leq 1} |z|^2  \nu^{\alpha,\beta}(dz)ds]^{\frac{1}{2}}+  \PX(\sup_{0\leq t\leq T}|X^{\alpha,\beta}_t|\geq M)\\
&\leq\frac{C\sqrt{\varepsilon}}{K} \sup_{\alpha \geq
\alpha_0}[\int_{ |z|\leq 1}
  |z|^2 \nu^{\alpha,\beta}(dz)]^{\frac{1}{2}}+\frac{1}{M^\theta}\sup_{\alpha \geq \alpha_0}
  \EX \sup_{0\leq t\leq T}|X^{\alpha,\beta}_t|^\theta.
\end{align*}
By choosing  sufficient large positive number $M$ in the last line of the above chain of inequalities so that
\[
\frac{1}{M^\theta}\sup\limits_{\alpha \geq \alpha_0} \EX \sup\limits_{0\leq t\leq
  T}|X^{\alpha,\beta}_t|^\theta
\]
is sufficient small.  Then by the fact that \  $\sup\limits_{\alpha \geq
\alpha_0}\int_{ |z|\leq 1}|z|^2 \nu^{\alpha,\beta}(dz)<\infty$\ , we deduce
that
\[
K_2 \rightarrow 0, \quad \text{as} \quad \varepsilon \rightarrow 0.
\]
  For $K_3$, we  have
\begin{align*}
   K_3
&=\sup_{\alpha \geq
\alpha_0}\PX(|\int_{\tau^\alpha}^{\tau^\alpha+\varepsilon}
 \int_{ |z|>1} z\sigma( X^{\alpha,\beta}_{s-}) N^{\alpha,\beta}(ds,dz)|>\frac{K}{3})\\
&\leq \sup_{\alpha \geq
\alpha_0}\PX(|\int_{\tau^\alpha}^{\tau^\alpha+\varepsilon}
 \int_{ |z|>1} z\sigma( X^{\alpha,\beta}_{s-}) {N}^{\alpha,\beta}(ds,dz)|>\frac{K}{3}; \sup_{0\leq t\leq T}|X^{\alpha,\beta}_t|\leq M) +  \PX(\sup_{0\leq t\leq T}|X^{\alpha,\beta}_t|\geq M)\\
&\leq \frac{C}{K} \sup_{\alpha
  \geq \alpha_0}\EX( \int_{\tau^\alpha}^{\tau^\alpha+\varepsilon} \int_{ |z|>
  1} |z| {N}^{\alpha,\beta}(ds,dz) )+  \PX(\sup_{0\leq t\leq T}|X^{\alpha,\beta}_t|\geq M)\\
&= \frac{C}{K} \sup_{\alpha
  \geq \alpha_0}\EX( \int_{\tau^\alpha}^{\tau^\alpha+\varepsilon} \int_{ |z|>
  1} |z| \nu^{\alpha,\beta}(dz)ds)+  \PX(\sup_{0\leq t\leq T}|X^{\alpha,\beta}_t|\geq M)\\
& \leq \frac{ C\varepsilon}{K} \sup_{\alpha \geq \alpha_0}
\int_{|z|\geq
  1}|z|\nu^{\alpha,\beta}(dz)+\frac{1}{M^\theta}\sup_{\alpha \geq \alpha_0} \EX \sup_{0\leq t\leq
  T}|X^{\alpha,\beta}_t|^\theta.
  \end{align*}
In the same way, by choosing  sufficient large positive number $M>0$ in the last line of the above chain of inequalities so that
\[
\frac{1}{M^\theta}\sup_{\alpha \geq \alpha_0} \EX \sup_{0\leq t\leq
  T}|X^{\alpha,\beta}_t|^\theta
\]
is sufficient small.  Then by the fact that  $ \sup\limits_{\alpha \geq
\alpha_0}\int_{|z|\geq1}|z|\nu^{\alpha,\beta}(dz) <\infty$, we deduce that
\[
K_3 \rightarrow 0, \quad \text{as} \quad \varepsilon \rightarrow 0.
\]
Combing above arguments for $K_1,K_2$ and $K_3$,  we get
\eqref{Aldous2}. Then, we can deduce that
the family $\{X^{\alpha,\beta}\}_{\alpha \geq \alpha_0}$ is tight in the
space $D_{J_1}([0,T],\mathbb{R}^d)$. \qed\end{proof}\bigskip

The following lemma characterizes the limit of $\mathcal{L}(X^\alpha)$, as the law of $X^\alpha$.

 \begin{lemma}\label{lem2.5}
If Assumption  \ref{asp1} is satisfied. Let
$\mu^{\alpha,\beta} :=\mathcal{L}(X^\alpha)$  and $\mu^*$ is the weak limit of any convergent subsequence
$\{\mu^{\alpha_n,\beta}\}$ $(\alpha_n \rightarrow 2$ as $n \rightarrow
\infty)$. Then $\mu^*$ is supported by $C_U([0,T],\mathbb{R}^d)$.
\end{lemma}
\begin{proof}
For any $\eta>0$ and $M>0$,  we have for any $ \theta$ in
$(0,1)$ that,
\begin{align*}
    \PX(\sup_{0\leq t\leq T}|X^{\alpha,\beta}_t-X^{\alpha,\beta}_{t-}| >\eta)
& \leq   \PX(\sup_{0\leq t\leq T}|X^{\alpha,\beta}_t-X^{\alpha,\beta}_{t-}| >\eta;
  \sup_{0\leq t\leq T}|X^{\alpha,\beta}_t|\leq M)\\
& \quad +  \PX(\sup_{0\leq t\leq T}|X^{\alpha,\beta}_t|\geq M)\\
&\leq\frac{c}{\eta}\EX \int_0^T \int_{|z|\geq \eta}|z| N^{\alpha,\beta}(ds,dz)+\frac{1}{M^\theta} \EX \sup_{0\leq t\leq T}|X^{\alpha,\beta}_t|^\theta\\
& \leq \frac{ c}{\eta}\int_{|z|\geq
\eta}|z|\nu^{\alpha,\beta}(dz)+\frac{1}{M^\theta} \EX \sup_{0\leq t\leq
T}|X^{\alpha,\beta}_t|^\theta.
\end{align*}
Letting first $M\rightarrow \infty$ and then $\alpha \rightarrow 2$,
we obtain that
\[
\lim_{\alpha \rightarrow 2}\PX(\sup_{0\leq t\leq
T}|X^{\alpha,\beta}_t-X^{\alpha,\beta}_{t-}|
>\eta) =0.
\]
Then $\mu^*$ is supported by
 $C_U([0,T],\mathbb{R}^d)$.
 \qed\end{proof}\bigskip

~\\

By using It\^{o}'s formula for the solution $X^{\alpha,\beta}_t$ of SDE \eqref{e:0}, for each $f(x)\in C^2(\mathbb{R}^d)$, we have
\begin{equation}\label{M1}
\begin{aligned}
&\quad \  f(X^{\alpha,\beta}_t)-f(x)\nonumber\\
&=\sum_{i=1}^{d}\int_0^t b_i(X^{\alpha,\beta}_s)\cdot \partial_i f(X^{\alpha,\beta}_s)ds\nonumber\\
&\quad+\sum_{i=1}^{d}\int_0^t \int_{ |z_i|\leq 1}[f(X^{\alpha,\beta}_{s-}+
\sigma_i(X^{\alpha,\beta}_{s-})z_i)-f(X^{\alpha,\beta}_{s-})]
\tilde{N}^{\alpha,\beta_i}(ds,dz_i)\nonumber\\
&\quad+\sum_{i=1}^{d}\int_0^t \int_{ |z_i|>1}[f(X^{\alpha,\beta}_{s-}+
\sigma_i(X^{\alpha,\beta}_{s-})z_i)-f(X^{\alpha,\beta}_{s-})]
N^{\alpha,\beta_i}(ds,dz_i)\nonumber\\
&\quad+\sum_{i=1}^{d}\int_0^t \int_{|z_i|\leq 1}[f(X^{\alpha,\beta}_{s-}+
\sigma_i(X^{\alpha,\beta}_{s-})z_i)-f(X^{\alpha,\beta}_{s-})- \sigma_i(X^{\alpha,\beta}_{s-})z_i\cdot
\nabla f(X^{\alpha,\beta}_{s-})] \nu^{\alpha,\beta_i}(dz_i)ds\nonumber\\
&=\tilde{M}_t+\mathcal{L}_{b,\sigma}^{\alpha,\beta} f(X_{s-}^{\alpha,\beta}),
\end{aligned}
\end{equation}
where  $\tilde{M}_t$  is a martingale and
$\mathcal{L}_{b,\sigma}^{\alpha,\beta} f(u)$ is defined as following:
\begin{align}\label{Lf-alpha}
\mathcal{L}_{b,\sigma}^{\alpha,\beta} f(u)
&:= b_i(u)\cdot \partial^i f(u)+\sum_{i=1}^{d}\int_{ |z_i|\leq 1}[f(u+
\sigma_i(u)z_i)-f(u)-\sigma_i(u)z_i\cdot \nabla f(u)]\nu^{\alpha,\beta_i}(dz_i)\\
&\quad+\sum_{i=1}^{d} \int_{ |z_i|>1}[f(u+
\sigma_i(u)z_i)-f(u)]\nu^{\alpha,\beta_i}(dz_i).\nonumber
\end{align}
Moreover, let
\begin{equation}\label{Lf}
    \mathcal{L}_{b,\sigma}f(u):=b_i(u)\cdot \partial^i f(u)+\text{Tr}\left( \sigma \sigma^*
    H(f)\right)(u),
\end{equation}
where  $\text{Tr}\left(
\sigma \sigma^* H(f)\right):= \sum_{i,j}(\sigma
\sigma^*)_{ij}\frac{\partial^2 f}{\partial x_i \partial x_j}$.
\begin{lemma}\label{lem2.6}
Let $\mathcal{L}_{b,\sigma}^{\alpha,\beta}$ and $ \mathcal{L}_{b,\sigma}$ be
the operators defined in \eqref{Lf-alpha} and \eqref{Lf},  respectively. Suppose
Assumption \ref{asp1} is verified, then for every $f\in
C^2_b (\R^d)$, we have $$\mathcal{L}_{b,\sigma}^{\alpha,\beta} f \rightarrow
\mathcal{L}_{b,\sigma} f$$ locally uniformly in $\mathbb{R}^d$ as
$\alpha$ tends to $2$.
\end{lemma}
\begin{proof}
For every $0<\delta<1$, we have
\begin{align}\label{g1}
\mathcal{L}_{b,\sigma}^{\alpha,\beta}
f(u)-\mathcal{L}_{b,\sigma}f(u)&=\sum_{i=1}^{d}\int_{ |z_i|\leq 1}[f(u+
\sigma_i(u)z_i)-f(u)-\sigma_i(u)z_i\cdot \nabla f(u)]\nu^{\alpha,\beta_i}(dz_i)\nonumber\\
&\quad+\sum_{i=1}^{d} \int_{ |z_i|>1}[f(u+
\sigma_i(u)z_i)-f(u)]\nu^{\alpha,\beta_i}(dz_i)-\text{Tr}\left( \sigma \sigma^*
    H(f)\right)(u)\nonumber\\
&=\sum_{i=1}^{d}\int_{ |z_i|\leq \delta}[f(u+
\sigma_i(u)z_i)-f(u)-\sigma_i(u)z_i\cdot \nabla f(u)]\nu^{\alpha,\beta_i}(dz_i)\nonumber\\
&\quad+\sum_{i=1}^{d} \int_{ |z_i|>\delta}[f(u+
\sigma_i(u)z_i)-f(u)]\nu^{\alpha,\beta_i}(dz_i)\nonumber \\
&\quad-\sum_{i=1}^{d}\int_{\delta < |z_i| \leq 1}\sigma_i(u)z_i\cdot\nabla f(u)\nu^{\alpha,\beta_i}(dz_i)
-\text{Tr}\left( \sigma \sigma^*H(f)\right)(u).\nonumber
\end{align}
For convenience, we denote
\begin{align}\nonumber
\mathcal{A}f(u)&:=\sum_{i=1}^{d}\int_{ |z_i|\leq \delta}[f(u+
\sigma_i(u)z_i)-f(u)-\sigma_i(u)z_i\cdot \nabla f(u)]\nu^{\alpha,\beta_i}(dz_i),\nonumber\\
\mathcal{B}f(u)&:=\sum_{i=1}^{d} \int_{ |z_i|>\delta}[f(u+
\sigma_i(u)z_i)-f(u)]\nu^{\alpha,\beta_i}(dz_i),\nonumber \\
\mathcal{C}f(u)&:=\sum_{i=1}^{d}\int_{\delta < |z_i| \leq 1}\sigma_i(u)z_i\cdot\nabla f(u)\  \nu^{\alpha,\beta_i}(dz_i), \nonumber\\
\mathcal{D}f(u)&:=\sum_{i=1}^{d}\int_0^1dt \int_0^t \int_{ |z_i|\leq
\delta}[\sigma_i(u)z_iH(f)(u)\cdot
  \sigma_i(u)z_i] \nu^{\alpha,\beta_i}(dz_i)ds.\nonumber
\end{align}
 By Taylor's formula, we get
$$
f(u+v)-f(u)-v \cdot \nabla f(u)=\int_0^1dt \int_0^t vH(f)(u+sv)\cdot
vds.
$$
Then, we can choose suitable $\delta<< 1$ such that
\begin{align}\nonumber
&\quad \mathcal{A}f(u)-\mathcal{D}f(u)\\
&=\sum_{i=1}^{d}
\int_{ |z_i|\leq \delta}\int_0^1dt \int_0^t
  \sigma_i(u)z_i[H(f)(u+\sigma_i(u)z_is)-H(f)(u)]\cdot \sigma_i(u)z_ids\nu^{\alpha,\beta_i}(dz_i)\nonumber
\end{align}
 is sufficient small. For a fixed $\delta>0$, we have
\begin{align*}
|\int_{|z_i|\le \delta}|z_i|^2 \nu^{\alpha,\beta_i}(dz_i)-2|\rightarrow 0, \ \ \  as \ \ \alpha\rightarrow 2.
\end{align*}
 Then
 \begin{center}
 \begin{align}\nonumber
&\quad \mathcal{D}f(u)-\text{Tr}\left( \sigma \sigma^*H(f)\right)(u)\nonumber\\
  &=\sum_{j,k=1}^d \sum_{i=1}^d\sigma_{ki}\sigma_{ji}f_{kj}(u)\int_0^1dt\int_0^tds(\int_{|z_i|\le \delta}|z_i|^2 \nu^{\alpha,\beta_i}(dz_i))-\sum_{j,k=1}^d \sum_{i=1}^d\sigma_{ki}\sigma_{ji}f_{kj}(u)\int_0^1dt\int_0^t 2ds\nonumber\\
  &=\sum_{j,k=1}^d \sum_{i=1}^d\sigma_{ki}\sigma_{ji}f_{kj}(u)\int_0^1dt\int_0^t ds(\int_{|z_i|\le \delta}|z_i|^2 \nu^{\alpha,\beta_i}(dz_i)-2)\rightarrow 0, \ \ \ as \ \ \alpha\rightarrow 2.\nonumber
 \end{align}
 \end{center}
Similarly, it can be shown that
 \begin{align}\nonumber
 |\mathcal{B}f(u)|\le C \sum_{i=1}^d \int_{|z_i|>\delta}|z_i| \nu^{\alpha,\beta_i}(dz_i)\rightarrow 0,
 \end{align}
 and
\begin{equation} \nonumber
|\mathcal{C}f(u)|
\le C\sum_{i=1}^{d}\int_{\delta<|z|\leq1} |z_i| \nu^{\alpha,\beta_i}(dz_i)
\le C\sum_{i=1}^{d}\int_{|z_i|>\delta}|z_i| \nu^{\alpha,\beta_i}(dz_i)\rightarrow 0,
\end{equation}
 as $\alpha$ tends to 2.
Then, we have
\begin{align}\nonumber
&\quad\mathcal{L}_{b,\sigma}^{\alpha,\beta}
f(u)-\mathcal{L}_{b,\sigma}f(u)\\\nonumber
&=\mathcal{A}f(u)+\mathcal{B}f(u)-\mathcal{C}f(u)-\text{Tr}\left( \sigma \sigma^*H(f)\right)(u)\nonumber\\
&=\mathcal{A}f(u)-\mathcal{D}f(u)+\mathcal{D}f(u)-\text{Tr}\left( \sigma \sigma^*H(f)\right)(u)+\mathcal{B}f(u)-\mathcal{C}f(u) \rightarrow 0,\nonumber
\end{align}
as $\alpha$ tends to 2.
\qed\end{proof}\bigskip

Now, by using the idea of Cerrai (see \cite[Theorem 6.2]{Cerrai}) and all of the lemmas presented in this section, we can prove Theorem \ref{thm1} as follows.
 ~\\

 \textbf{\emph{Proof of Theorem 2.1.}}\ \ \
Due to sequence $\{\mu^{\alpha_n,\beta}\}_{n\in \mathbb{N}}$ converges weakly to $\mu^*$. If we are able to identify $\mu^*$ with $\mathcal{L}(X)$, where $\mathcal{L}(X)$ is the distribution of $X$, then we conclude that the whole sequence $\{\mu^{\alpha_n,\beta}\}_{n\in \mathbb{N}}$ weakly convergent to $\mathcal{L}(X)$ in $D_{J_1}([0,T],\mathbb{R}^d)$.
Let $\mathcal{L}_{b,\sigma}^{\alpha,\beta}$ and $ \mathcal{L}_{b,\sigma}$ be the operators defined in \eqref{Lf-alpha} and \eqref{Lf}, respectively. For any  $f\in C_b^2(\mathbb{R}^d)$, we obtain that the process
\begin{align*}
 f(X_t^{\alpha,\beta}) - f(x) - \int_0^t \mathcal{L}_{b,\sigma}^{\alpha,\beta}f(X_t^{\alpha,\beta})
\end{align*}
is a $\mathcal{F}_t-$martingale.

We denote by $\EX^{\mu^\alpha}$ as the expectation with respect to the probability measure $\mu^\alpha$ and $\EX^{\mu^*}$ as the expectation with respect to the probability measure $\mu^*$. We denote by $\eta(t)$ the canonical process in $D_{J_1}([0,T],\mathbb{R}^d)$. For every $0\le s_0 < s_1 < \cdots < s_m \le s < t$ and $f_0, f_1, \cdots, f_m \in C_b (\mathbb{R}^d)$, we have
\begin{equation}\label{M2}
\EX^{\mu^{\alpha}}\left[ \left( f(\eta_t)-f(\eta_s)-\int_s^t \mathcal{L}_{b,\sigma}^{\alpha,\beta}f(\eta_r)dr  \right) f_0(\eta_{s_0})\cdots f_m(\eta_{s_m})\right]=0.
\end{equation}
By \eqref{M2} , we have
\begin{align*}
&\quad \EX^{\mu^*}\left [ \left(f(\eta_t)-f(\eta_s)-\int_s^t\mathcal{L}_{b,\sigma}f(\eta_r)dr\right) f_0(\eta_{s_0})\cdots f_m(\eta_{s_m})\right]\\
& = \lim_{n \rightarrow \infty}\EX^{\mu^{\alpha_n}} \left[ \left(f(\eta_t)-f(\eta_s)-\int_s^t\mathcal{L}_{b,\sigma}f(\eta_r)dr\right) f_0(\eta_{s_0})\cdots f_m(\eta_{s_m})\right]\\
& = \lim_{n \rightarrow \infty}\EX^{\mu^{\alpha_n}} \Big[ \left(f(\eta_t)-f(\eta_s)-\int_s^t\mathcal{L}_{b,\sigma}f(\eta_r)dr\right) f_0(\eta_{s_0})\cdots f_m(\eta_{s_m})\\
& \quad \quad \quad\quad\quad\quad -\left(f(\eta_t)-f(\eta_s)-\int_s^t \mathcal{L}_{b,\sigma}^{\alpha_n,\beta} f(\eta_r)dr\right) f_0(\eta_{s_0})\cdots f_m(\eta_{s_m})\Big]\\
& = \lim_{n \rightarrow \infty}\EX^{\mu^{\alpha_n}} \left[ \left(\int_s^t \mathcal{L}_{b,\sigma}^{\alpha_n,\beta} f(\eta_r)-\mathcal{L}_{b,\sigma}f(\eta_r) dr \right) f_0(\eta_{s_0})\cdots f_m(\eta_{s_m}) \right].
\end{align*}
By Lemma \ref{lem2.4}, Lemma \ref{lem2.5} and Lemma \ref{lem2.6}, choosing a suitable $M>0$, we obtain
\begin{align}\label{M4}
&\quad\lim_{n\rightarrow \infty} \EX^{\mu^{\alpha_n}}\left[ \left(\int_s^t \mathcal{L}_{b,\sigma}^{\alpha_n, \beta}f(\eta_r)-\mathcal{L}_{b,\sigma}f(\eta_r) dr\right) f_0(\eta_{s_0})\cdots f_m(\eta_{s_m})\right]\nonumber \\
&\le\lim_{n\rightarrow \infty}  C\EX^{\mu^{\alpha_n}}\left[\int_s^t | \mathcal{L}_{b,\sigma}^{\alpha_n, \beta}f(\eta_r)-\mathcal{L}_{b,\sigma}f(\eta_r) |dr \right] \nonumber \\
&\le \lim_{n\rightarrow \infty}  C\EX^{\mu^{\alpha_n}}\left[\int_s^t | \mathcal{L}_{b,\sigma}^{\alpha_n, \beta}f(\eta_r)-\mathcal{L}_{b,\sigma}f(\eta_r) |dr; \sup_{s \le r \le t}|\eta_r| \le M\right] \nonumber \\
&\quad+\lim_{n \rightarrow \infty}  C\EX^{\mu^{\alpha_n}}\left[\int_s^t |\mathcal{L}_{b,\sigma}^{\alpha_n, \beta}f(\eta_r)-\mathcal{L}_{b,\sigma}f(\eta_r) |dr; \sup_{s\le r \le t}|\eta_r|> M\right] \nonumber \\
&=0.
\end{align}
Hence, according to \eqref{M4}, we can conclude that
\begin{align*}
 \EX^{\mu^*}\left[\left(f(\eta_t)-f(\eta_s)-\int_s^t\mathcal{L}_{b,\sigma}f(\eta_r)dr\right) f_0(\eta_{s_0})\cdots f_m(\eta_{s_m})\right]=0.
\end{align*}
This means that $f(\eta_t)-f(\eta_0)-\int_0^t\mathcal{L}_{b,\sigma}f(\eta_s)ds$ is a martingale under $\mu^*$. By using the argument is similar to  \cite[Theorem 6.2]{Cerrai}, we say that a probability measure $\mu^*$ solves the martingale problem which induces a weak solution of the SDE \eqref{e2}, associated with $\mathcal{L}_{b,\sigma}$. Thus, we can conclude that $\mu^*=\mathcal{L}(X)$.
\qed\bigskip

 \begin{remark}
Now, we have shown  the solution  of an SDE driven by a non-symmetric $\alpha$-stable process  converged weakly to that  of an SDE driven by  a Brownian motion. From the proof of Theorem \ref{thm1}, for any $\delta>0$, we observe that
 \begin{align*}
 \int_0^. \int_{|z|>1} zN^{\alpha,\beta}(dt,dz) \rightarrow 0,  \ \   \int_0^. \int_{|z|\le 1} z\tilde{N}^{\alpha,\beta}(dt,dz) \rightarrow B. , \ as\  \alpha\rightarrow 2
 \end{align*}
weakly in $D_{J_1}([0,T],\mathbb{R}^d)$, where the process \  $ \int_0^. \int_{|z|\le 1} z\tilde{N}^{\alpha,\beta}(dt,dz)$\  is  the compensated sum
of small jumps, and  is said to be ``central part'', which plays a vital part. $ \int_0^. \int_{|z|>1} zN^{\alpha,\beta}(ds,dz)$\  describing the ``large jumps'', and is said to be ``non-symmetric part''. However, in case of non-symmetric noise, the process \ $ \int_0^. \int_{|z|\le 1} z\tilde{N}^{\alpha,\beta}(dt,dz) $  is non-symmetric, which converges to a Brownian motion in $D_{J_1}([0,T],\mathbb{R}^d)$ and in law when $\alpha$ is close to 2. Moreover, the SDE driven by   $ \int_0^. \int_{|z|\le 1} z\tilde{N}^{\alpha,\beta}(dt,dz)$  is similar to the SDE driven by Brownian motion in $D_{J_1}([0,T],\mathbb{R}^d)$ and in law.
 \end{remark}

\section{\bf Proof of Theorem 2.2}  \label{Weak Convergence}

In this section we will derive the rate of weak convergence stated in Theorem \ref{thm2}. For this purpose we will use the backward Kolmogorov equation.
Let's consider the equation
\begin{equation}\label{Keq}
\begin{aligned}
&(\partial_t+\mathcal{L}_{b,\sigma})u(t,x)=0, t\in[0,T]\\
&u(T,x)=f(x),
\end{aligned}
\end{equation}
where
\begin{align*}
\mathcal{L}_{b,\sigma} u(t,x)= b_i(x)\partial^iu(t,x)+\text{Tr}\left( \sigma \sigma^*H_x\big(u(t,x)\big)\right),
\end{align*}
which denotes the generator of $X_t$.
\begin{lemma}\label{lem4.1}
If Assumption \ref{asp2} is satisfied, then for each $f\in  {C}^{2,\gamma}(\mathbb{R}^d)$, the equation \eqref{Keq} exists a unique solution $u\in  {C}^{2,\gamma}(H)$.
\end{lemma}

A proof of this result can be found in R. Mikulevicius \cite[Theorem 4]{rate}.  In order to establish the rate of weak convergence under  Assumption \ref{asp2}, we need to prove weak convergence firstly. Since this provement is quite similar to the argument of Theorem \ref{thm1},  we only establish a uniform estimate in the following lemma.

\begin{lemma}\label{lem4.2}
If  Assumption \ref{asp2} is satisfied, then for
every $ \theta \in (0,\alpha_0)$,
\begin{equation}\label{solution1}
    \sup_{\alpha\geq\alpha_0} \EX [\sup_{0\leq t\leq T} |X^{
    \alpha,\beta}_t|^\theta] <\infty,
\end{equation}
where $\alpha_0 \in (\frac{3}{2}, 2).$
\end{lemma}
\begin{proof}
We need only to prove \eqref{solution1} on condition that $\theta \in [1,\alpha_0)$ because the case where $\theta \in (0,1)$ is an immediate corollary by H${\rm \ddot{o}}$lder's inequality.  In fact, we have
\begin{align*}
X_t^{\alpha,\beta}&=x+\int_0^t b(X_s^{\alpha,\beta})ds +\int_0^t\int_{|z|\le 1}\sigma(X_{s-}^{\alpha,\beta})z \tilde{N}^{\alpha,\beta}(ds,dz)
+\int_0^t\int_{|z|>1}\sigma(X_{s-}^{\alpha,\beta})zN^{\alpha,\beta}(ds,dz)\\
&=x+\int_0^t b(X_s^{\alpha,\beta})ds +\int_0^t\int_{|z|\le1}\sigma(X_{s-}^{\alpha,\beta})z \tilde{N}^{\alpha,\beta}(ds,dz)\\
&\quad\ +\int_0^t\int_{|z|>1}\sigma(X_{s-}^{\alpha,\beta})z\nu^{\alpha,\beta}dz+\int_0^t\int_{|z|>1}\sigma(X_{s-}^{\alpha,\beta})z\tilde{N}^{\alpha,\beta}(ds,dz).
\end{align*}
First, there exists a positive constant $M$, such that\ $|b(x)|\le M$ and $|\sigma(x)|\le M$, for all $x\in\mathbb{R}^d$.  We obtain
\begin{align}\label{E1}
&\quad\EX(\sup_{0\le t \le T}|\int_0^tb(X_s^{\alpha,\beta})ds|^{\theta})
\le M^\theta T^\theta < \infty.
\end{align}
 By using Burkholder's inequality \cite[Theorem 26.12]{12}, and Lemma \ref{lem2.2}, we have
\begin{align}\label{E2}
&\quad\sup_{\alpha\geq \alpha_0}\EX\sup_{0\le t \le T}|\int_0^t\int_{|z|\le1} \sigma(X_{s-}^{\alpha,\beta})z\tilde{N}^{\alpha,\beta}(ds,dz)|^\theta \nonumber \\
&\le C_\theta \sup_{\alpha\geq\alpha_0}\EX|\int_0^T\int_{|z|\le 1} |\sigma(X_{s-}^{\alpha,\beta})z|^2 \nu^{\alpha,\beta}(dz)ds|^\frac{\theta}{2} \nonumber \\
&\le C_\theta M^\theta T^\frac{\theta}{2}\sup_{\alpha\geq \alpha_0}|\int_{|z|\le 1} |z|^2 \nu^{\alpha,\beta}(dz)|^\frac{\theta}{2} <\infty.
\end{align}
By applying \cite[Lemma 2.3]{Liu}, we obtain
\begin{align}\label{E3}
&\quad\sup_{\alpha\geq \alpha_0}\EX\sup_{0\le t \le T}|\int_0^t\int_{|z|>1} \sigma(X_{s-}^{\alpha,\beta})z\tilde{N}^{\alpha,\beta}(ds,dz)|^\theta \nonumber \\
&\le \sup_{\alpha\geq \alpha_0}C\EX \left(\int_0^T\int_{|z|>1} |\sigma(X_{s-}^{\alpha,\beta})z|^\theta \nu^{\alpha,\beta}(dz)ds\right)
+ \sup_{\alpha\geq \alpha_0}C\EX \left(\int_0^T\int_{|z|>1}|\sigma(X_{s-}^{\alpha,\beta})z| \nu^{\alpha,\beta}(dz)ds\right)^\theta \nonumber \\
&\le C M^\theta\sup_{\alpha\geq\alpha_0}[T\int_{|z|>1} |z|^\theta \nu^{\alpha,\beta}(dz) + T^\theta(\int_{|z|>1}|z|\nu^{\alpha,\beta}(dz))^\theta] < \infty
\end{align}
and
\begin{align}\label{E4}
&\quad\sup_{\alpha\geq\alpha_0}\EX\sup_{0\le t \le T} |\int_0^t\int_{|z|>1}\sigma(X_{s-}^{\alpha,\beta})z\nu^{\alpha,\beta}(dz)ds|^\theta\le T^\theta M^{\theta} \sup_{\alpha\geq\alpha_0}\left(\int_{|z|>1}|z|\nu^{\alpha,\beta}(dz)\right)^\theta < \infty.
\end{align}
By using $(x+y)^\theta \le C_{\theta}(x^\theta +y^\theta)$ for $\theta\in(1,2)$, we have
\begin{equation}\label{E5}
\begin{aligned}
 \quad \sup_{\alpha\geq\alpha_0} \EX [\sup_{0\leq t\leq T} |X^{
    \alpha,\beta}_t|^\theta]
    \le C[\sup_{\alpha\geq\alpha_0}\EX(&\sup_{0\le t \le T}|\int_0^tb(X_s^{\alpha,\beta})ds|^{\theta} \\
    \quad+\sup_{\alpha\geq \alpha_0}\EX&\sup_{0\le t \le T}|\int_0^t\int_{|z|\le1} \sigma(X_{s-}^{\alpha,\beta})z\tilde{N}^{\alpha,\beta}(ds,dz)|^\theta  \\
    \quad+\sup_{\alpha\geq \alpha_0}\EX&\sup_{0\le t \le T}|\int_0^t\int_{|z|>1} \sigma(X_{s-}^{\alpha,\beta})z\tilde{N}^{\alpha,\beta}(ds,dz)|^\theta  \\
    \quad+\sup_{\alpha\geq\alpha_0}\EX&\sup_{0\le t \le T} |\int_0^t\int_{|z|>1}\sigma(X_{s-}^{\alpha,\beta})z\nu^{\alpha,\beta}(dz)ds|^\theta].
\end{aligned}
\end{equation}
By formulas \eqref{E1}-\eqref{E5}, we obtain
\begin{align*}
 \quad \sup_{\alpha\geq\alpha_0} \EX [\sup_{0\leq t\leq T} |X^{
    \alpha,\beta}_t|^\theta] < \infty.
\end{align*}
The proof is complete.
\qed\bigskip\end{proof}

Now, we proceed to the proof of Theorem \ref{thm2}.

~\\
 \textbf{\emph{Proof of Theorem 2.2.}}\ \ \
Let $\mathcal{L}^{\alpha,\beta}_{b,\sigma}$ be the generator of $X_t^{\alpha,\beta}$,
\begin{align*}
 \mathcal{L}^{\alpha,\beta}_{b,\sigma} u(t,x)&=b_i(x)\partial^i u(t,x)+\sum_{i=1}^d\int_{|z_i|>1} u(t,x+\sigma_i(x)z_i)-u(t,x) \ \nu^{\alpha,\beta_i}(dz_i)\\
&\quad +\sum_{i=1}^d\int_{|z_i|\le 1} u(t,x+\sigma_i(x)z_i)-u(t,x)-\sigma_i(x)z_i \cdot\nabla u(t,x) \nu^{\alpha,\beta_i}(dz_i).
\end{align*}
According to Lemma \ref{lem4.1}, there exists a unique solution $u(t,x)\in {C}^{2,\gamma}(H)$ to \eqref{Keq}, for each $f(x)\in  {C}^{2,\gamma}(\mathbb{R}^d)$. By using It\^{o} formula,  we have
\begin{align*}
\EX[f(X_t^{\alpha,\beta})-f(X_t)]&=\EX[u(t,X_t^{\alpha,\beta})-u(0,x_0)]\\
&=\EX [\int_0^t \partial_tu(s,X_s^{\alpha,\beta})+ \mathcal{L}_{b,\sigma}^{\alpha,\beta}u(s,X_s^{\alpha,\beta}) ds]\\
&= \EX[\int_0^t -\mathcal{L}_{b,\sigma}u(s,X_s^{\alpha,\beta})+ \mathcal{L}_{b,\sigma}^{\alpha,\beta} u(s,X_s^{\alpha,\beta})ds ].
\end{align*}
Notice that $\sigma$ and $b$ is bounded, then for any $\delta>0$ we have
\begin{align*}
&\quad\ \int_0^t(\mathcal{L}_{b,\sigma}^{\alpha,\beta} -\mathcal{L}_{b,\sigma}) u(s,X_s^{\alpha,\beta})ds\\
&=\sum_{i=1}^d\int_0^t\int_{|z_i|\le 1}  u(s,X_{s-}^{\alpha,\beta} +\sigma_i(X_{s-}^{\alpha,\beta})z_i) - u(s,X_{s-}^{\alpha,\beta})-\sigma_i(X_{s-}^{\alpha,\beta})z_i\cdot\partial_x u(s,X_{s-}^{\alpha,\beta}) \nu^{\alpha,\beta_i}(dz_i)ds\\
&\quad -\int_0^tTr[\sigma\sigma^*(X_{s-}^{\alpha,\beta})H_x\big(u(s,X_{s-}^{\alpha,\beta})\big)] ds\\
&\quad + \sum_{i=1}^d\int_0^t\int_{|z_i|>1}  u(s,X_{s-}^{\alpha,\beta} +\sigma_i(X_{s-}^{\alpha,\beta})z_i) - u(s,X_{s-}^{\alpha,\beta}) \nu^{\alpha,\beta_i}(dz_i)ds\\
&= \sum_{i=1}^d\int_0^t\int_{|z_i|\le \delta}  u(s,X_{s-}^{\alpha,\beta} +\sigma_i(X_{s-}^{\alpha,\beta})z_i) - u(s,X_{s-}^{\alpha,\beta})-\sigma_i(X_{s-}^{\alpha,\beta})z_i\cdot\partial_x u(s,X_{s-}^{\alpha,\beta}) \nu^{\alpha,\beta_i}(dz_i)ds \\
&\quad -\int_0^tTr[\sigma\sigma^*(X_{s-}^{\alpha,\beta})H_x\big(u(s,X_{s-}^{\alpha,\beta})\big)]ds \\
&\quad + \sum_{i=1}^d\int_0^t\int_{|z_i|>\delta}  u(s,X_{s-}^{\alpha,\beta} +\sigma_i(X_{s-}^{\alpha,\beta})z_i)
 - u(s,X_{s-}^{\alpha,\beta}) \nu^{\alpha,\beta_i}(dz_i)ds\\
&\quad - \sum_{i=1}^{d}\int_0^t\int_{\delta\le|z_i|\le1}\sigma_i(X_{s-}^{\alpha,\beta})z_i\cdot\partial_xu(s,X_{s-}^{\alpha,\beta}) \nu^{\alpha,\beta_i}(dz_i)ds.
\end{align*}
For convenience, we denote
\begin{align*}\nonumber
\mathcal{A}u(s,X_{s}^{\alpha,\beta})&=\sum_{i=1}^{d}\int_{ |z_i|\leq \delta}u(s,X_{s-}^{\alpha,\beta} +\sigma_i(X_{s-}^{\alpha,\beta})z_i) - u(s,X_{s-}^{\alpha,\beta})-\sigma_i(X_{s-}^{\alpha,\beta})z_i\cdot\partial_x u(s,X_{s-}^{\alpha,\beta}) \nu^{\alpha,\beta_i}(dz_i),\nonumber\\
\mathcal{B}u(s,X_{s}^{\alpha,\beta})&=\sum_{i=1}^{d} \int_{ |z_i|>\delta} u(s,X_{s-}^{\alpha,\beta} +\sigma_i(X_{s-}^{\alpha,\beta})z_i)
 - u(s,X_{s-}^{\alpha,\beta})\nu^{\alpha,\beta_i}(dz_i),\nonumber \\
\mathcal{C}u(s,X_{s}^{\alpha,\beta})&=\sum_{i=1}^{d}\int_{\delta < |z_i| \leq 1}\sigma_i(X_{s-}^{\alpha,\beta})z_i
\cdot\partial_xu(s,X_{s-}^{\alpha,\beta})\nu^{\alpha,\beta_i}(dz_i), \nonumber\\
\mathcal{D}u(s,X_{s}^{\alpha,\beta})&=\sum_{i=1}^{d}\int_0^1dr \int_0^r \int_{ |z_i|\leq
\delta}[\sigma_i(X_{s-}^{\alpha,\beta})z_iH_x\big(u(s,X_{s-}^{\alpha,\beta})\big)\cdot
 \sigma_i(X_{s-}^{\alpha,\beta})z_i] \nu^{\alpha,\beta_i}(dz_i)ds.\nonumber
\end{align*}
By Taylor's formula, we get
$$
u(t,x+y)-u(t,x)-y \cdot \nabla_x u(t,x)=\int_0^1dr \int_0^r yH_x\big(u(t,x+\xi y)\big)\cdot
yd\xi.
$$
Then, we have
\begin{align*}\nonumber
&\quad |\EX\int_0^t\mathcal{A}\big(u(s,X_{s}^{\alpha,\beta})\big)-\mathcal{D}\big(u(s,X_{s}^{\alpha,\beta})\big)ds|\\
&\le\sum_{i=1}^{d}\EX\int_0^t
\int_{ |z_i|\leq \delta}\int_0^1dr \int_0^r
  \big|\sigma_i(X_{s-}^{\alpha,\beta})z_i\Big(H_x\big(u(s,X_{s-}^{\alpha,\beta}+\sigma_i(X_{s-}^{\alpha,\beta})z_i\xi)\big)\nonumber\\
  &\quad\quad-H_x\big(u(s,X_{s-}^{\alpha,\beta})\big)\Big)\cdot \sigma_i(X_{s-}^{\alpha,\beta})z_i\big|d\xi \nu^{\alpha,\beta_i}(dz_i)ds\nonumber \\
  &\le C\sum_{i=1}^{d}\EX\int_0^t
\int_{ |z_i|\leq \delta}\int_0^1dr \int_0^r
|z_i|^2\big|H_x\big(u(s,X_{s-}^{\alpha,\beta}+\sigma_i(X_{s-}^{\alpha,\beta})z_i\xi)\big)
  -H_x\big(u(s,X_{s-}^{\alpha,\beta})\big)\big|d\xi \nu^{\alpha,\beta_i}(dz_i)ds.
\end{align*}
 Since $u\in {C}^{2,\gamma}(H)$, we can choose suitable $\delta<< 1$ such that $\EX\int_0^t|\mathcal{A}\big(u(s,X_{s}^{\alpha,\beta})\big)-\mathcal{D}\big(u(s,X_{s}^{\alpha,\beta})\big)|ds$ is sufficient small. For a fixed $\delta>0$, by using Lemma~\ref{lem2.2}, we obtain
\begin{equation}\nonumber
|\int_{|z_i|\le \delta}|z_i|^2 \nu^{\alpha,\beta_i}(dz_i)-2|\le C(2-\alpha).
\end{equation}
 Then, we have
 \begin{center}
 \begin{align*}\nonumber
&\quad\ |\EX\int_0^t \mathcal{D}\big(u(s,X_s^{\alpha,\beta})\big)-\text{Tr}\left( \sigma \sigma^*(X_{s-}^{\alpha,\beta})H_x\big(u(s,X_{s-}^{\alpha,\beta})\big)\right)ds|\nonumber\\
  &\le\sum_{j,k=1}^d \sum_{i=1}^d\int_0^t|\sigma_{ki}\sigma_{ji}\partial_{kj}^2u(s,X_{s-}^{\alpha,\beta})\int_0^1dr\int_0^r d\xi(\int_{|z_i|\le \delta}|z_i|^2 \nu^{\alpha,\beta_i}(dz_i)-2)|ds\nonumber\\
  &\le C (2-\alpha).
 \end{align*}
 \end{center}
Furthermore, we have
 \begin{align*}\nonumber
 |\mathcal{B}\big(u(s,X_s^{\alpha,\beta})\big)|\le c \sum_{i=1}^d \int_{|z_i|>\delta}|z_i| \nu^{\alpha,\beta_i}(dz_i)\le C(2-\alpha),
 \end{align*}
 and
 \begin{align*} \nonumber
\quad|\EX\int_0^t\ \mathcal{C}\big(u(s,X_{s}^{\alpha,\beta})\big)\ ds|&\le C\sum_{i=1}^{d}\int_{\delta<|z_i|\leq1} |z_i| \nu^{\alpha,\beta_i}(dz_i)\\
&\le C\sum_{i=1}^{d}\int_{|z_i|>\delta}|z_i| \nu^{\alpha,\beta_i}(dz_i)\le C(2-\alpha).
\end{align*}
Thus, we obtain
\begin{align*}\nonumber
&\quad\ \big|\EX\int_0^t \mathcal{L}_{b,\sigma}^{\alpha,\beta}
\big(u(s,X_s^{\alpha,\beta})\big)-\mathcal{L}_{b,\sigma}\big(u(s,X_s^{\alpha,\beta})\big) \ ds\big|\nonumber\\
&=\big|\EX\int_0^t(\mathcal{A}+\mathcal{B}-\mathcal{C})\big(u(s,X_s^{\alpha,\beta})\big)-\text{Tr}\left( \sigma \sigma^*(X_{s-}^{\alpha,\beta})H_x\big(u(s,X_{s-}^{\alpha,\beta})\big)\right)\ ds\big|\nonumber\\
&=\big|\EX\int_0^t(\mathcal{A}-\mathcal{D})\big(u(s,X_s^{\alpha,\beta})\big)ds\big|\\
&\quad+\big|\EX\int_0^t\mathcal{D}\big(u(s,X_s^{\alpha,\beta})\big)-\text{Tr}\left( \sigma \sigma^*(X_{s-}^{\alpha,\beta})H_x\big(u(s,X_s^{\alpha,\beta})\big)\right)ds\big|\\
&\quad+\big|\EX\int_0^t\mathcal{B}\big(u(s,X_s^{\alpha,\beta})\big)ds\big|+\big|\EX\int_0^t\mathcal{C}\big(u(s,X_s^{\alpha,\beta})\big)ds\big| \\
&\le C(2-\alpha).
\end{align*}
The proof is complete. \qed\bigskip

\begin{remark}
The estimate \eqref{wr} implies that the weak convergence rate is $2-\alpha$, which is the optimal rate. Please see the following example for illustrating it.
\end{remark}

\begin{example}\label{example}
Let $L_t^\alpha$ and $B_t$ are one-dimensional symmetric $\alpha$-stable process ($\beta=0$) and standard Brownian motion, respectively. Then, according to \cite{moment}, we have
\begin{align*}
\EX|L_t^\alpha|=\frac{\Gamma(1-\frac{1}{\alpha})}{\int_0^\infty u^{-2}sin^2u du}\sqrt{t}.
\end{align*}
By using $\Gamma(1-x)\Gamma(x)=\frac{\pi}{sin(\pi x)}, \ x\in(0,1)$  \ and \  $\Gamma(\frac{1}{2})=\sqrt{\pi}$,  we have
\begin{align*}
\lim_{\alpha\rightarrow 2}\Big|\frac{\EX|L_t^\alpha|-\EX|B_t|}{2-\alpha}\Big|&=\lim_{\alpha\rightarrow 2}\Big|\frac{\Gamma(1-\frac{1}{\alpha})(\int_0^\infty u^{-2}sin^2u\  du )^{-1}- \frac{2}{\sqrt{\pi}}}{2-\alpha}\Big|\sqrt{t}\\
&=\lim_{\alpha\rightarrow 2}\Big|\frac{\frac{2\pi}{\Gamma(\frac{1}{\alpha})sin(\frac{\pi}{\alpha})}-2\sqrt{\pi}}{\pi(2-\alpha)}\Big|\sqrt{t}\\
&=\lim_{\alpha\rightarrow 2}\Big|\frac{2\big(\sqrt{\pi}-\Gamma(\frac{1}{\alpha})sin\frac{\pi}{\alpha} \ \big)}{\pi(2-\alpha)}\Big|\sqrt{t}\\
&=\lim_{\alpha\rightarrow 2}\Big| \frac{2}{\pi \alpha^2} \big( \Gamma'(\frac{1}{\alpha})sin\frac{\pi}{\alpha}+ \pi\Gamma(\frac{1}{\alpha})cos\frac{\pi}{\alpha}\ \big)\Big|\sqrt{t}\\
&=\frac{|\Gamma'(\frac{1}{2})|}{2\pi}\sqrt{t}\\
&=\frac{\gamma+2ln2}{2\sqrt{\pi}}\sqrt{t},
\end{align*}
where $\gamma$ is Euler-Mascheroni constant. Hence, for every\  $t\in[0,T]$,  we can obtain that there is a constant $C$, such that
\begin{equation}\label{optimal}
\big|\ \EX|L_t^\alpha|-\EX|B_t| \ \big| \geq C(2-\alpha).
\end{equation}
The estimate \eqref{optimal}  implies that the optimal weak convergence rate is $2-\alpha$.
\end{example}

\section{\bf Appendix: Proof of Lemma \ref{lem2.2} } \label{Appendix}

In this appendix we give the proof of Lemma \ref{lem2.2}.\\
\begin{proof} Let $\alpha \geq l$,
\begin{align*}\nonumber
\int_{|z|\le\delta} |z|^2 \nu^{\alpha,\beta}(dz)=\int_0 ^\delta \frac{ C_1 z^2}{|z|^{\alpha+1}} dz
+\int_{-\delta}^0 \frac{C_2 z^2}{|z|^{\alpha+1}} dz
=\frac{K_\alpha \delta^{2-\alpha}}{2-\alpha}, \nonumber
\end{align*}
\begin{align*} \nonumber
\int_{|z|>\delta}|z|^{\vartheta}\nu^{\alpha,\beta}(dz)=\int_{z>\delta} \frac{K_\alpha (1+\beta)}{2z^{\alpha+1-\vartheta}}dz + \int_{z< {-\delta}}\frac{K_\alpha(1-\beta)}{2z^{\alpha+1-\vartheta}}dz
=\frac{K_\alpha \delta^{\vartheta-\alpha}}{\alpha-\vartheta}. \nonumber
\end{align*}
Then by using\ $\lim\limits_{x \rightarrow 0}x\Gamma(x)=\lim\limits_{x\rightarrow 0} \frac{\pi x}{\Gamma(1-x) sin \pi x}=1$, we have
\begin{align*} \nonumber
\lim_{\alpha \rightarrow 2}\int_{ |z|\leq \delta}|z|^2\nu^{\alpha,\beta}(dz)
=\lim_{\alpha \rightarrow 2}\frac{\alpha(1-\alpha) \delta^{2-\alpha}}{(2-\alpha)\Gamma(2-\alpha)cos(\frac{\pi\alpha}{2})}=2, \nonumber
\end{align*}
\begin{align*} \nonumber
\lim_{\alpha \rightarrow 2} \int_{|z|>\delta}|z|^\vartheta \nu^{\alpha,\beta}(dz)
=\lim_{\alpha \rightarrow 2}\frac{2(2-\alpha) \delta^{\vartheta-\alpha}}{\alpha-\vartheta}=0. \nonumber
\end{align*}
Fix $\alpha\in(0,2)$,
\begin{align*} \nonumber
 \lim_{\delta\rightarrow 0+}\delta^{\alpha-2}\int_{ |z|\leq \delta} |z|^{2}
  \nu^{\alpha,\beta}(dz)= \lim_{\delta\rightarrow 0+}\frac{K_\alpha}{2-\alpha}<\infty. \nonumber
\end{align*}
Hence $ \lim\limits_{\delta\rightarrow 0+}\int_{ |z|\leq \delta} |z|^{2}
  \nu^{\alpha,\beta}(dz)=0. $

 Furthermore, by using L'Hospital's rule and $\lim\limits_{x \rightarrow 0} \frac{x\Gamma(x)}{x-1}=\lim\limits_{x\rightarrow 0}(x\Gamma(x))'=-1$, we have
\begin{align*}
\lim_{\alpha \rightarrow 2}\frac{\big|\int_{ |z|\leq \delta} |z|^{2}
  \nu^{\alpha,\beta}(dz)-2\big|}{2-\alpha}=\lim_{\alpha \rightarrow 2} \big|\frac{\alpha(1-\alpha) \delta^{2-\alpha}-2(2-\alpha)\Gamma(2-\alpha)cos(\frac{\pi\alpha}{2})}{(2-\alpha)}\big|=C.
\end{align*}
This result implies that
\begin{align*}
|\int_{ |z|\leq \delta} |z|^{2}
  \nu^{\alpha,\beta}(dz)-2|\le C(2-\alpha).
  \end{align*}
And let $\vartheta=1$, we get
\begin{align*}
 \int_{|z|>\delta}|z| \nu^{\alpha,\beta}(dz)\le C(2-\alpha).
\end{align*}
The proof is complete.\qed\end{proof}\bigskip

\section*{Ackonwledgement}
The authors ackonwledge support provided by NNSF of China (Nos. 11971186, 11971367), Science and Technology
Research Projects of Hubei Provincial Department of Education No. B2022077.


\begin{thebibliography}{99}
\bibitem{Aldous}D. Aldous, Stopping times and tightness, The Annals of Probability, 6 (2) (1978) 335-340.
\bibitem{Bill}P. Billingsley, Convergence of Probability Measure, 2nd edition, Wiley, New York, 1999.
\bibitem{Cerrai}S. Cerrai, A Khasminskii type averaging principle for stochastic reaction-diffusion equations, The Annals of Applied Probability, 19 (3) (2009) 899-948.
\bibitem{moment}C. D. Hardin Jr.,  Skewed stable variables and processes. Technical Report 79, Univ. North Carolina, Chapel Hill, 1984.
\bibitem{17}J. Jacod and A. N. Shiryaev, Limit theorems for stochastic processes, Grundlehren der mathematischen Wissenschaften, 2nd edition, Springer Verlag, Berlin, 2009.
\bibitem{12}O. Kallenberg, Foundations of Modern Probability, Springer Science\&Business Media, 2006.
\bibitem{Liu}X. Liu, On the $\alpha$-dependence of stochastic differential equations with H$\ddot{o}$lder drift and driven by $\alpha$-stable L\'{e}vy processes, Journal of Mathematical Analysis and Applications, 506 (1) (2022) 125642.
\bibitem{rate}R. Mikulevicius, On the rate of convergence of simple and jump-adapted weak Euler schemes for L\'{e}vy driven SDEs, Stochastic Processes and their Applications, 122 (7) (2012) 2730-2757.
\bibitem{rate1}R. Mikulevicius, C. Zhang, On the rate of convergence of weak Euler approximation for nondegenerate SDEs driven by L\'{e}vy processes, Stochastic Processes and their Applications, 121 (8) (2011), 1720-1748.
\bibitem{14}D. Pollard, Convergence of Stochastic Processes, Springer Science Business Media, 2006.
\bibitem{Sato}K.-I. Sato, L\'{e}vy Processes and Infinite Divisibility, Cambridge University Press (1999).
\bibitem{16}A. V. Skorohod, Limit theorems for stochastic processes, Theory of Probability \& Its Applications, 1 (3) (1956) 261-260.
\bibitem{zhu}F. Xi and C. Zhu, Jump type stochastic differential equations with non-Lipschitz coefficients: Non-confluence, Feller and strong Feller properties, and exponential ergodicity, Journal of Differential Equations 266 (8) (2019) 4668-4711.
\bibitem{SGI}L. Xie and X. Zhang, Ergodicity of stochastic differential equations with jumps and singular coefficients,
Annales de l'Institut Henri Poincar\'{e}, Probabilit\'{e}s et Statistiques, 56 (1) (2020) 175-229.
\end{thebibliography}
\end{document}